\documentclass[12pt]{amsart}
\usepackage{amsfonts,amssymb,amscd,amsmath,enumerate,verbatim}
\usepackage[latin1]{inputenc}
\usepackage{amscd}
\usepackage{latexsym}
\usepackage{multicol}
\usepackage[dvipdfmx]{graphicx}
\input xy

\xyoption{all}

\def\NZQ{\mathbb}               
\def\NN{{\NZQ N}}

\def\frk{\mathfrak}               

\def\Phi{{\frk N}}
\def\opn#1#2{\def#1{\operatorname{#2}}} 
\opn\chara{char} 
\opn\length{\ell} 
\opn\pd{pd} 
\opn\rk{rk}
\opn\projdim{proj\,dim} 
\opn\injdim{inj\,dim} 
\opn\rank{rank}
\opn\depth{depth} 
\opn\grade{grade} 
\opn\height{height}
\opn\embdim{emb\,dim} 
\opn\codim{codim}

\opn\Tr{Tr} 
\opn\bigrank{big\,rank}
\opn\superheight{superheight}
\opn\lcm{lcm}
\opn\trdeg{tr\,deg}
\opn\reg{reg} 
\opn\lreg{lreg} 
\opn\ini{in} 
\opn\lpd{lpd}
\opn\size{size}
\opn\mult{mult}
\opn\dist{dist}
\opn\cone{cone}
\opn\lex{lex}
\opn\rev{rev}
\opn\div{div} \opn\Div{Div} \opn\cl{cl} \opn\Cl{Cl}
\opn\Spec{Spec} \opn\Supp{Supp} \opn\supp{supp} \opn\Sing{Sing}
\opn\Ass{Ass} \opn\Min{Min}
\opn\Ann{Ann} \opn\Rad{Rad} \opn\Soc{Soc}
\opn\Syz{Syz} \opn\Im{Im} \opn\Ker{Ker} \opn\Coker{Coker}
\opn\Am{Am} \opn\Hom{Hom} \opn\Tor{Tor} \opn\Ext{Ext}
\opn\End{End} \opn\Aut{Aut} \opn\id{id} \opn\ini{in}

\opn\nat{nat}
\opn\pff{pf}
\opn\Pf{Pf} \opn\GL{GL} \opn\SL{SL} \opn\mod{mod} \opn\ord{ord}
\opn\Gin{Gin}
\opn\Hilb{Hilb}\opn\adeg{adeg}\opn\std{std}\opn\ip{infpt}
\opn\Pol{Pol}
\opn\sat{sat}
\opn\Var{Var}
\opn\Gen{Gen}

\opn\aff{aff} \opn\con{conv} \opn\relint{relint} \opn\st{st}
\opn\lk{lk} \opn\cn{cn} \opn\core{core} \opn\vol{vol}
\opn\link{link} \opn\star{star}
\opn\gr{gr}

\def\Ac{{\mathcal A}}
\def\Bc{{\mathcal B}}

\def\Rc{{\mathcal R}}

\def\pot#1#2{#1[\kern-0.28ex[#2]\kern-0.28ex]}

\opn\dirlim{\underrightarrow{\lim}}
\opn\inivlim{\underleftarrow{\lim}}
\let\Dirsum=\bigoplus

\let\to=\rightarrow

\def\Implies{\ifmmode\Longrightarrow \else
        \unskip${}\Longrightarrow{}$\ignorespaces\fi}
\def\implies{\ifmmode\Rightarrow \else
        \unskip${}\Rightarrow{}$\ignorespaces\fi}
\def\iff{\ifmmode\Longleftrightarrow \else
        \unskip${}\Longleftrightarrow{}$\ignorespaces\fi}

\let\:=\colon
\newtheorem{Theorem}{Theorem}[section]

\newtheorem{Corollary}[Theorem]{Corollary}

\let\epsilon\varepsilon
\let\phi=\varphi
\let\kappa=\varkappa
\def\qed{\ifhmode\textqed\fi
      \ifmmode\ifinner\quad\qedsymbol\else\dispqed\fi\fi}
\def\textqed{\unskip\nobreak\penalty50
       \hskip2em\hbox{}\nobreak\hfil\qedsymbol
       \parfillskip=0pt \finalhyphendemerits=0}
\def\dispqed{\rlap{\qquad\qedsymbol}}

\opn\dis{dis}
\opn\height{height}
\opn\dist{dist}
\opn\supp{supp}
\def\pnt{{\raise0.5mm\hbox{\large\bf.}}}

\opn\Lex{Lex}

\begin{document}
	
\title{Ideals with componentwise linear powers}

	\author {Takayuki Hibi and Somayeh Moradi}

\address{Takayuki Hibi, Department of Pure and Applied Mathematics,
	Graduate School of Information Science and Technology, Osaka
	University, Suita, Osaka 565-0871, Japan}
\email{hibi@math.sci.osaka-u.ac.jp}

\address{Somayeh Moradi, Department of Mathematics, Faculty of Science, Ilam University,
	P.O.Box 69315-516, Ilam, Iran}
\email{so.moradi@ilam.ac.ir}

\dedicatory{ }
\thanks{Takayuki Hibi is partially supported by JSPS KAKENHI 19H00637. Somayeh Moradi is supported by the Alexander von Humboldt Foundation.}

\begin{abstract}
Let $S=K[x_1,\ldots,x_n]$ be the polynomial ring over a field $K$, and let $A$ be a finitely generated standard graded $S$-algebra. We show that if the defining ideal of $A$ has a quadratic initial ideal, then all the graded components of $A$ are componentwise linear. Applying this result to the Rees ring $\mathcal{R}(I)$ of a graded ideal $I$ gives a criterion on $I$ to have componentwise linear powers.  Moreover, for any given graph $G$, a construction on $G$ is presented  which produces graphs whose cover ideals $I_G$ have componentwise linear powers. This in particular implies that for any Cohen-Macaulay Cameron-Walker graph $G$ all powers of $I_G$ have linear resolutions. Moreover, forming a cone on special graphs like unmixed chordal graphs, path graphs and Cohen-Macaulay bipartite graphs  produces cover ideals with componentwise linear powers.
\end{abstract}

\maketitle

\setcounter{tocdepth}{1}
\section*{Introduction}	

Componentwise linear ideals were first introduced by Herzog and the first author of this paper in \cite{HH1}. Since their introduction
they have emerged as an intriguing class of ideals deserving special attention, due to some of their interesting characterizations in combinatorics and commutative algebra, see  \cite{AHH1,HV,HRW,R1,Y}. One research theme in this context is to find ideals whose all powers are componentwise linear or have linear resolutions. 
Let $S=K[x_1,\ldots,x_n]$ be the polynomial ring over a field $K$. A graded ideal $I\subset S$  is called {\em componentwise linear} if for each integer $j$ the ideal generated by all homogeneous elements of degree $j$ in $I$ has a linear resolution. 
In this paper we mainly consider the question: what hypotheses on $I$  ensure that all powers of $I$ are componentwise linear? In particular when $I$ is the cover ideal of a graph $G$, what graph constructions lead to cover ideals with componentwise linear powers? 
To investigate this question it is natural to consider the  Rees algebra $\mathcal{R}(I)$. For a graded ideal $I\subset S$ let $J\subset T=S[y_1,\ldots,y_m]$ be the defining ideal of $\mathcal{R}(I)$. The ideal $J$ is said to satisfy the $x$-condition with respect to a monomial order $<$ on $T$ if any minimal monomial generator of $\ini_<(J)$
is of the form $vw$ with $v\in S$ of degree at most one and $w\in K[y_1,\ldots,y_m]$.
This property was first defined in \cite{HHZ}.
When $J$ has the $x$-condition property, then it is proved in 
\cite[Corollary 1.2]{HHZ} that if $I$ is equigenerated, then each power of $I$ has a linear resolution. It is natural to ask when $I$ is not equigenerated,  how the $x$-condition affects the powers of $I$. This is considered in \cite{HHM}, where it is shown that if $J$ satisfies the $x$-condition with respect to some special monomial order, 
 then $I^k$ has linear quotients with respect to some set of generators for all $k$, which may not be minimal. But once it is minimal, it implies the componentwise linearity of $I^k$, see \cite[Theorem 8.2.15]{HH}. 
When $I$ is a monomial ideal and  $\ini_<(J)$ is generated by quadratic monomials,  it is shown  in \cite[Theorem 3.6]{HHM} that this set of generators is minimal and hence all powers of $I$ are componentwise linear. In the first section of this paper we extend this result to any graded ideal $I$, see Theorem \ref{shameo} and Corollary \ref{extend}. 

The question of having linear or componentwise linear powers has attracted special attention for the ideals arising from graphs and it has been studied in several papers.   
Some families of such ideals whose powers inherit componentwise linear or linear property are
cover ideals of Cohen-Macaulay bipartite graphs \cite{HHbi},  unmixed chordal graphs \cite[Theorem 2.7]{HHO}, chordal graphs that are  $(C_4, 2K_2)$-free \cite[Theorem 3.7]{N},  trees  \cite[ Corollary 3.5]{KK},
path graphs, biclique graphs and Cameron-Walker graphs whose bipartite graph is a complete bipartite graph  \cite[Corollary 4.7]{HHM} and edge ideals with linear resolutions \cite[Theorem 3.2]{HHZ}. In Section 2 of this paper we consider a construction on a graph $G$  denoted by  $G(H_1, \ldots, H_n)$ which attaches to each vertex $x_i$ of $G$ a graph $H_i$.   As the main result we show in Theorem \ref{construction} that if for each $i$, the defining ideal $J_{H_i}$ of $\mathcal{R}(I_{H_i})$ 
has a quadratic Gr\"obner basis, then each power of the vertex cover ideal $I_{G(H_1, \ldots, H_n)}$ is componentwise linear. To this aim we first show in Theorem \ref{graphjoin}
that if each $J_{H_i}$ satisfies the $x$-condition , then $J_{G(H_1, \ldots, H_n)}$ satisfies the $x$-condition with respect to some monomial order.
Cohen-Macaulay Cameron-Walker graph and cone graphs are examples of such constructions.

\section{Algebras with componentwise linear graded components}	
Let $K$ be a field and let $A=\Dirsum_{i,j}A_{(i,j)}$ be a bigraded $K$-algebra with $A_{(0,0)}=K$. Set $A_j=\Dirsum_iA_{(i,j)}$. We assume that $A_0$ is the polynomial ring $S= K[x_1,\ldots,x_n]$ with the standard grading.
Then $A=\Dirsum_jA_j$ is a graded $S$-algebra,   and each $A_j$ is a graded $S$-module with grading  $(A_j)_i=A_{(i,j)}$ for all $i$.
We assume in addition that $A$  is a finitely generated standard graded $S$-algebra and $(A_1)_i=0$ for $i<0$.

We fix a system of homogeneous generators $f_1,\ldots,f_m$ of $A_1$ with $\deg f_i=d_i$ for $i=1,\ldots,m$.
Let $T=K[x_1,\ldots,x_n,y_1,\ldots,y_m]$  be the bigraded polynomial ring   over $K$ with the grading induced by 
$\deg x_i=(1,0)$ for $i=1,\ldots,n$ and $\deg y_j=(d_j,1)$  for  $j=1,\ldots,m$.
Define the $K$-algebra homomorphism $\varphi\: T\to A$ with $\varphi(x_i)=x_i$ for $i=1,\ldots,n$ and $\varphi(y_j)=f_j$ for $j=1,\ldots, m$.  Then $\varphi$ is a surjective $K$-algebra homomorphism of bigraded $K$-algebras, and hence  $J=\Ker(\varphi)$ is a bigraded ideal in $T$.
The defining ideal $J$ of $A$ is said to satisfy the {\em  $x$-condition}, with respect to a monomial order $<$ on $T$  if all $u\in \mathcal{G}(\ini_<(J))$ are of the form $vw$ with $v\in S$ of degree $\leq 1$ and
$w\in K[y_1,\ldots,y_m]$. Here $\mathcal{G}(I)$ denotes the minimal set of monomial generators of a monomial ideal $I$.

Given a monomial  order  $<'$ on the polynomial ring $K[y_1,\ldots,y_m]$  and a monomial order $<_x$ on $K[x_1,\ldots,x_n]$, let $<$ be a monomial order on $T$ such that

\begin{eqnarray}
	\label{monomialorder}
	\prod_{i=1}^n x_i^{a_i}\prod_{i=1}^m y_i^{b_i} &<& \prod_{i=1}^n x_i^{a'_i} \prod_{i=1}^m y_i^{b'_i},
\end{eqnarray}
if
\begin{eqnarray*}
	\prod_{i=1}^m y_i^{b_i} <'\prod_{i=1}^m y_i^{b'_i}\quad \text{or}\quad  \prod_{i=1}^m y_i^{b_i}&=&\prod_{i=1}^m y_i^{b'_i}\quad  \text{and} \quad
	\prod_{i=1}^n x_i^{a_i} <_x \prod_{i=1}^n x_i^{a'_i}.
\end{eqnarray*}
We call $<$ the order induced by the orders $<'$ and  $<_x$.

A graded $S$-module $M$ has {\em linear quotients},  if there exists a system of homogeneous generators $f_1,\ldots,f_m$  of $M$ with the property that each of the colon ideals $(f_1,\ldots,f_{j-1}):f_j$ is generated by linear forms. 
With the above notation and terminology we have


\begin{Theorem}
	\label{shameo}
	Let $J$ be the defining ideal of the $K$-algebra $A$. If $\ini_{<}(J)$ is generated by quadratic monomials, then for all $k\geq 1$, the $S$-module $A_k$ has linear quotients with respect to a minimal generating set and hence it is componentwise linear.
\end{Theorem}

\begin{proof}
	Let $k\geq 1$ be an integer. For any element $h=f_{i_1}\cdots f_{i_k}\in A_k$, we set
	$h^{*}=y_{i_1}\cdots y_{i_k}$. Let $\{h_1^{*},\ldots,h_{d}^*\}$ be the set of all standard monomials of bidegree $(*,k)$ in $T$ which are of the form $h^*$.  Then by \cite[Lemma~3.2]{HHM}, $h_1,\ldots,h_d$ is a system of generators of $A_k$. We show that this is indeed a minimal system of generators of $A_k$. Since $f_1,\ldots,f_m$ is a minimal generating set of $A_1$, there is nothing to prove for $k=1$. Now, let $k>1$.
	By contradiction suppose that there exists an integer $j$ such that $h_j=\sum_{\ell\neq j}p_{\ell}h_{\ell}$ for some homogeneous polynomials $p_{\ell}\in S$.
	Then $q=h_j^*-\sum_{\ell\neq j}p_{\ell}h_{\ell}^*\in J$.
	Without loss of generality assume that $q$ has the least initial term among all the expressions of the form $h_j^*-\sum_{\ell\neq j}p'_{\ell}h_{\ell}^*\in J$. Since $h_j^*$ is a standard monomial, we have  $\ini_<(q)=\ini_<(p_i)h_i^*\in\ini_{<}(J)$ for some $i\neq j$ with $p_i\neq 1$.
	By assumption there exists an element $g$ in the Gr\"{o}bner basis of $J$ such that $\ini_<(g)$ has degree two and $\ini_<(g)$ divides $\ini_<(p_i)h_i^*$.  If $\ini_<(g)=y_ry_t$ for some $r$ and $t$, then $y_ry_t$ divides $h_i^*$, which implies that $h_i^*\in \ini_<(J)$, a contradiction.
	So $\ini_<(g)=x_ry_t$ for some $r$ and $t$ and then $x_r$ divides $\ini_<(p_i)$ and $y_t$ divides $h_i^*$. Let $g=x_ry_t-\sum_{\ell=1}^s c_{\ell}u_{\ell}y_{j_{\ell}}$ for some $c_{\ell}\in K$ and monomials $u_{\ell}\in S$. Then
	\begin{equation}\label{-1}
		x_rf_t=\sum_{\ell=1}^s c_{\ell}u_{\ell}f_{j_{\ell}}.
	\end{equation} 
	We have $u_{\ell}\neq 1$ for all $\ell$. Indeed, if $u_{\lambda}=1$ for some $\lambda$, then
	\begin{equation}\label{0}
		f_{j_{\lambda}}=c_{\lambda}^{-1}x_rf_t-\sum_{\ell\neq\lambda} c_{\lambda}^{-1}c_{\ell}u_{\ell}f_{j_{\ell}}.
	\end{equation} 
	This is a contradiction, since $f_1,\ldots,f_m$ is a minimal generating set of $A_1$. Thus $u_{\ell}\neq 1$ for all $\ell$.
	We have
	\begin{eqnarray*}
		\ini_<(p_i)h_i=(\ini_<(p_i)/x_r)x_rf_t(h_i/f_t)= \\
		(\ini_<(p_i)/x_r)(\sum_{\ell=1}^s c_{\ell}u_{\ell}f_{j_{\ell}}) (h_i/f_t).
	\end{eqnarray*}
	
	Hence $\ini_<(p_i)h_i=\sum_{\ell=1}^s c_{\ell}w_{\ell}h'_{\ell}$, for the monomials $w_{\ell}=(\ini_<(p_i)/x_r)u_{\ell}\neq 1$ and elements $h'_{\ell}=f_{j_{\ell}}(h_i/f_t)\in A_k$. Let $p'_i=p_i-\ini_<(p_i)$. Then
	\begin{equation}\label{1}
		h_j=(\ini_<(p_i)+p'_i)h_i+\sum_{\ell\neq i,j}p_{\ell}h_{\ell}=\sum_{\ell=1}^s c_{\ell}w_{\ell}h'_{\ell}+p'_ih_i+\sum_{\ell\neq i,j}p_{\ell}h_{\ell}.
	\end{equation}
	
	We show that $h'_{\ell}\neq h_j$ for all $1\leq\ell\leq s$. Indeed, by the equality (\ref{-1}) and that $u_{\ell}\neq 1$, we have $\deg(f_t)\geq \deg(f_{j_{\ell}})$. Hence $\deg(h'_{\ell})\leq \deg(h_i)$. While
	$\deg(h_j)=\deg(p_i)+\deg(h_i)>\deg(h_i)$. Hence $\deg(h_j)>\deg(h'_{\ell})$ which implies that $h'_{\ell}\neq h_j$.
	Also $y_{j_{\ell}}<'y_t$ implies that $(h'_{\ell})^*<'h_i^*$. By~(\ref{1}), we have
	\begin{equation}\label{2}
		q'=h_j^*-\sum_{\ell=1}^s c_{\ell}w_{\ell}(h'_{\ell})^*-p'_ih_i^*-\sum_{\ell\neq i, j}p_{\ell}h_{\ell}^*\in J.
	\end{equation}
	Since $(h'_{\ell})^*\neq h_j^*$ for all $\ell$ and $\ini_<(q')<\ini_<(q)$, by the assumption on $q$ we conclude that there exists at least one integer $1\leq\ell\leq s$ such that the monomial $(h'_{\ell})^*$ is not standard.
	By~\cite[Lemma 2.2]{HHM}, for any such $\ell$ there exist standard monomials $h_{t_{\ell,1}}^*,\ldots,h_{t_{\ell,b}}^*$ with $h_{t_{\ell,1}}^*<\cdots<h_{t_{\ell,b}}^*<(h'_{\ell})^*$ and homogeneous polynomials $v_{\ell,\lambda}$ such that
	\begin{equation}\label{3}
		(h'_{\ell})^*-\sum_{\lambda=1}^b v_{\ell,\lambda}h_{t_{\ell,\lambda}}^*\in J.
	\end{equation}
	Again notice that $j\notin\{t_{\ell,1},\ldots,t_{\ell,b}\}$, since $\deg(h_j)>\deg(h'_\ell)\geq\deg(h_{t_{\ell,\lambda}})$ for any $1\leq \lambda\leq b$. 
	Set $I_1=\{\ell: 1\leq \ell\leq s, \ (h'_{\ell})^*\ \textrm{is not standard}\}$ and $I_2=[s]\setminus I_1$. By (\ref{2}) and (\ref{3}) we obtain an expression
	$$q''=h_j^*-\sum_{\ell\in I_1}\sum_{\lambda=1}^b c_{\ell}w_{\ell}v_{\ell,d}\ h_{t_{\ell,d}}^*-\sum_{\ell\in I_2}c_{\ell}w_{\ell}(h'_{\ell})^*-p'_ih_i^*-\sum_{\ell\neq i, j}p_{\ell}h_{\ell}^*\in J$$ in terms of standard monomials. Since
	$\ini_<(q'')<\ini_<(q)$, we get a contradiction.
	
	Thus $h_1,\ldots,h_d$ is a minimal system of generators of $A_k$. Now using~\cite[Theorem 2.3]{HHM}, we conclude that for each $k$, $A_k$ has linear quotients with respect to its minimal set of generators. Thus it follows from ~\cite[Theorem 8.2.15]{HH} that $A_k$ is componentwise linear.
\end{proof}

Applying Theorem~\ref{shameo} to the Rees algebra of a graded ideal we obtain the following corollary which generalizes \cite[Corollary 10.1.8]{HH} and \cite[Theorem 3.6]{HHM}. 

\begin{Corollary}
	\label{extend}
	Let $I\subset S$ be a graded ideal and let $J$ be the defining ideal of the Rees algebra $\Rc(I)$. If $\ini_{<}(J)$ is generated by quadratic monomials with respect to the monomial order defined in {\em (\ref{monomialorder})}, then for any $k\geq 1$, $I^k$ has linear quotients with respect to its minimal monomial generating set and hence it is componentwise linear.
\end{Corollary}

\section{cover ideals of graphs with componentwise linear powers}	

Let $G$ be a finite simple graph on the vertex set $V(G) = \{x_1, \ldots, x_n \}$ and let $E(G)$ be the set of edges of $G$.  A subset $C\subseteq V(G)$ is called a \emph{vertex cover} of $G$,
when it intersects any edge of $G$. Moreover, $C$ is called a \emph{minimal vertex cover} of $G$, if it is a vertex cover and no proper subset of $C$ is a vertex cover of $G$. Let as before $S = K[x_1, \ldots, x_n]$ denote the polynomial ring in $n$ variables over a field $K$.  We associate each subset $C \subset V(G)$ with the monomial $u_C = \prod_{x_i \in C} x_i$ of $S$.  Let $C_1, \ldots, C_q$ denote the minimal vertex covers of $G$. The {\em cover ideal} of $G$ is defined as $I_G = (u_{C_1}, \ldots, u_{C_q})$. 
The Rees algebra of $I_G$ is the toric ring
\[
\Rc(I_G) = K[x_1, \ldots, x_n, u_{C_1}t, \ldots, u_{C_q}t] \subset S[t].
\] 
Let $T = S[y_1, \ldots, y_q]$ denote the polynomial ring and define the surjective map $\pi : T \to \Rc(I_G)$ by setting $\pi(x_i) = x_i$ and $\pi(y_j) = u_{C_j}t$.  The toric ideal $J_G \subset T$ of $\Rc(I_G)$ is the kernel of $\pi$.  Let $<_{\rm lex}$ denote the pure lexicographic order on $S$ induced by the ordering $x_1 > \cdots > x_n$ and suppose that $u_{C_q} <_{\rm lex} \cdots <_{\rm lex} u_{C_1}$.  Let $<'_{\rm lex}$ denote the pure lexicographic order on $K[y_1, \ldots, y_q]$ induced by the ordering $y_1 > \cdots > y_q$.  We let $<^\sharp$ be  the monomial order on $T$  which is induced by the orders  $<'_{\rm lex}$ and  $<_{\rm lex}$.  

Given finite simple graphs $H_1, \ldots, H_n$ with $V(H_i) = \{z^{(i)}_1, \ldots, z^{(i)}_{r_i} \}$, we construct the graph $G(H_1, \ldots, H_n)$ on $V(G) \cup V(H_1) \cup \cdots \cup V(H_n)$ whose set of edges is
\[
E(G) \cup E(H_1) \cup \cdots \cup E(H_n) \cup \left(\,\bigcup_{\substack{1 \leq i \leq n \\ 1 \leq j \leq r_i}}\{x_i, z^{(i)}_j \}\right)
\] 

\begin{Theorem}\label{graphjoin}
Suppose that each $J_{H_i}$ satisfies the $x$-condition with respect to $z^{(i)}_1 > \cdots > z^{(i)}_{r_i}$.  Then $J_{G(H_1, \ldots, H_n)}$ satisfies the $x$-condition with respect to 
\[
z^{(1)}_1 > z^{(1)}_2 > \cdots > z^{(1)}_{r_1} > z^{(2)}_{1} > \cdots > z^{(2)}_{r_2} > \cdots > z^{(n)}_{r_n} > x_1 > \cdots > x_n.
\]
\end{Theorem}

\begin{proof} Let $S = K[x_1, \ldots, x_n, z^{(1)}_1, \ldots, z^{(n)}_{r_n} ]$ denote the polynomial ring in $n + r_1 + \cdots + r_n$ variables over a field $K$.  Let $C^{(i)}_1, \ldots, C^{(i)}_{s_i}$ denote the minimal vertex covers of $H_i$.  Given a vertex cover $C$ of $G$, we introduce a subset $C' = C \cup C^{(i)} \cup \cdots \cup C^{(n)}$ of $V(G) \cup V(H_1) \cup \cdots \cup V(H_n)$ for which
\[
C^{(i)} = \left\{
\begin{array}{lll}
V(H_i) & \text{if} & x_i \not\in C, \\
C^{(i)}_j & \text{if} & x_i \in C,
\end{array}
\right.
\]
where $1 \leq j \leq s_i$ is arbitrary.  It follows that $C'$ is a minimal vertex cover of $G$ and that every minimal vertex cover of $G$ is of the form $C'$.  

Let $<_{\rm lex}$ denote the pure lexicographic order on $S$ with respect to the above ordering $z^{(1)}_1 > \cdots > x_n$.  Let $C_1, \ldots, C_q$ denote the minimal vertex covers of $G(H_1, \ldots, H_n)$ and suppose that $u_{C_q} <_{\rm lex} \cdots <_{\rm lex} u_{C_1}$.  

Let $T_i = K[z^{(i)}_1, \ldots, z^{(i)}_{r_i}, y^{(i)}_1, \ldots, y^{(i)}_{s_i}]$ denote the polynomial ring in $r_i + s_i$ variables over $K$ and $J_{H_i} \subset T_i$ the toric ideal of $\Rc(I_{H_i})$ which is the kerner of $\pi_i : T_i \to \Rc(I_{H_i})$.  Let $w = z^{(i)}_j w''$ be a monomial belonging to the minimal system of monomial generators of ${\rm in}_{<^\sharp}(J_{H_i})$, where $w''$ is a monomial in $y^{(i)}_1, \ldots, y^{(i)}_{s_i}$.  Let $\pi_i(w) = z^{(i)}_j u_{C^{(i)}_{\xi_1}}t \cdots u_{C^{(i)}_{\xi_a}}t$.  Let $C_{\zeta_{i'}}$ be a minimal vertex covers of $G(H_1, \ldots, H_n)$ with $C_{\zeta_{i'}} \cap V(H_i) = C^{(i)}_{\xi_{i'}}$ for each $1 \leq i' \leq a$.  It follows that $z^{(i)}_j y_{\zeta_1} \cdots y_{\zeta_a}$ belongs to ${\rm in}_{<^\sharp}(J_{G(H_1, \ldots, H_n)}) \subset  T = S[y_1, \ldots, y_q]$.

Let $C$ be a minimal vertex cover of $G(H_1, \ldots, H_n)$ with $x_i \not\in C$ and  
\[
C' = ((C \cup \{x_i\}) \setminus V(H_i) ) \cup C^{(i)}_j,
\]   
where $1 \leq j \leq s_i$ is arbitrary.  Then $C'$ is a minimal vertex cover of $G(H_1, \ldots, H_n)$ with $u_{C'} <_{\rm lex} u_{C}$.  Let $C = C_e$ and $C' = C_f$.  Then $e < f$ and $$x_i y_e - \prod_{z^{(i)}_{j'} \not\in C^{(i)}_j} z^{(i)}_{j'} y_f \in J_{G(H_1, \ldots, H_n)}$$ whose initial monomial is $x_i y_e$.

Let $\Ac$ denote the set of monomials of the form either $z^{(i)}_j y_{\zeta_1} \cdots y_{\zeta_a}$ or $x_i y_e$ constructed above.  Let $\Bc$ denote the set of monomials in $y_1, \ldots, y_q$ belonging to the minimal system of monomial generators of ${\rm in}_{<^\sharp}(J_{G(H_1, \ldots, H_n)})$.

Let $(\Ac, \Bc)$ denote the monomial ideal of $T$ generated by $\Ac \cup \Bc$.  One claims that ${\rm in}_{<^\sharp}(J_{G(H_1, \ldots, H_n)}) = (\Ac, \Bc)$.  One must prove that, for monomials $u$ and $v$ of $T$ not belonging to $(\Ac, \Bc)$ with $u \neq v$, one has $\pi(u) \neq \pi(v)$.  Let $u = u_x u_z u_y$ and $v = v_x v_z v_y$ with $u \neq v$, where $u_x, v_x$ are monomials in $x_1, \ldots, x_n$, where $u_z, v_z$ are monomials in $z^{(i)}_j, 1 \leq i \leq n, 1 \leq j \leq r_i$, and where $u_y, v_y$ are monomials in $y_1, \ldots, y_s$.  Suppose that $u$ and $v$ are relatively prime and that $\pi(u) = \pi(v)$.  

Let, say, $u_x \neq 1$ and $x_i$ divide $u_x$.  Since $x_i$ does not divide $v_x$ and since $\deg u_y = \deg v_y$, it follows that there is $y_a$ which divids $u_y$ for which $x_i \not\in C_a$.  Hence $x_i y_a \in \Ac$, a contradiction.  Thus $u_x = v_x = 1$.  

Let $\pi(u) = u_z \cdot u_{C_{a_1}}t \cdots u_{C_{a_p}}t$ and $\pi(v) = v_z \cdot u_{C_{a'_1}}t \cdots u_{C_{a'_p}}t$.  Let, say, $u_z \neq 1$ and $z^{(i)}_j$ divide $u_z$.  In each of $\pi(u)$ and $\pi(v)$, replace each of $x_{1}, \ldots, x_n$ with $1$ and replace $z^{(i')}_j$ with $1$ for each $i' \neq i$ and for each $1 \leq j \leq r_i$.  Then $\pi(u)$ comes to $$u'_z \cdot u_{C^{(i)}_{c_1}}t \cdots u_{C^{(i)}_{c_{p'}}}t\,\left(\prod_{j=1}^{r_i}z_j^{(i)}\right)^{p-p'}$$ and $\pi(v)$ comes to $$v'_z \cdot u_{C^{(i)}_{c'_1}}t \cdots u_{C^{(i)}_{c'_{p'}}}t\,\left(\prod_{j=1}^{r_i}z_j^{(i)}\right)^{p-p'},$$ where each of $u'_z$ and $v'_z$ is a monomial in $z_{1}^{(i)}, \ldots, z_{r_i}^{(i)}$.  One has $p' > 0$ and $$u'_z \cdot y^{(i)}_{c_1} \cdots y^{(i)}_{c_{p'}} - v'_z \cdot y^{(i)}_{c'_1} \cdots y^{(i)}_{c'_{p'}} \in J_{H_i}.$$  Since $z^{(i)}_j$ divides $u'_z$, one has $u'_z \cdot y^{(i)}_{c_1} \cdots y^{(i)}_{c_{p'}} - v'_z \cdot y^{(i)}_{c'_1} \cdots y^{(i)}_{c'_{p'}} \neq 0$ and its initial monomial belongs to ${\rm in}_{<^\sharp}(J_{H_i})$.  Since $J_{H_i}$ satisfies the $x$-condition, it follows that either $u \in (\Ac, \Bc)$ or $v \in (\Ac, \Bc)$, a contradiction.  Thus $u_z = v_z = 1$.

Since $u = u_y, v = v_y, u - v \neq 0$ and $\pi(u) = \pi(v)$, one has either $u \in (\Bc)$ or $v \in (\Bc)$, a contradiction.
\, \, \, \, \, \, \, \, \, \,  
\, \, \, \, \, \, \, \, \, \,  
\, \, \, \, \, \, \, \, \, \, 
\, \, \, \, \, \, \, \,  
\end{proof}

\begin{Theorem}\label{construction}
Suppose that each $J_{H_i}$ has a quadratic Gr\"obner basis with respect to some monomial order.  Then each power of the vertex cover ideal $I_{G(H_1, \ldots, H_n)}$ possesses an order of linear quotients on its minimal monomial generating set, and hence it is componentwise linear.
\end{Theorem}

\begin{proof}
We keep the notation used in the proof of Theorem~\ref{graphjoin}.
Suppose that $J_{H_i}$ has a quadratic Gr\"obner basis with respect to $z^{(i)}_1 > \cdots > z^{(i)}_{r_i}$ for each $i$. 
By Theorem~\ref{graphjoin}, the ideal $J=J_{G(H_1, \ldots, H_n)}$ satisfies the $x$-condition with respect to the order $<$ induced by the orders $z^{(i)}_1 > \cdots > z^{(i)}_{r_i}$. Hence by the proof of \cite[Theorem 3.3]{HHM}, for any positive integer $k$, the ideal $(I_{G(H_1, \ldots, H_n)})^k$  has a system of generators $h_1,\ldots,h_s$,  each of them of the form $h_i=u_{C_{i_1}}\cdots u_{C_{i_k}}$, which possess an order of linear quotients $h_1<\cdots<h_s$ and such that $h_i^*=y_{i_1}\cdots y_{i_k}$ is a standard monomial of $T$ with respect to $<^\sharp$ and $J$.
We prove that $\{h_1,\ldots,h_s\}$ is the minimal generating set of monomials of $(I_{G(H_1, \ldots, H_n)})^k$. 

Suppose that $h_j=wh_i$ for some integers $i$ and $j$ and a monomial $w\in S$. We should show that $w=1$ and $i=j$.
Let $h_i=u_{C_{i_1}}\cdots u_{C_{i_k}}$ and $h_j=u_{C_{j_1}}\cdots u_{C_{j_k}}$.  If $w=1$ and $i\neq j$, then  $h_j^*-h_i^*\in J$. So either $h_i^*$ or $h_j^*$ belongs to $\ini_{<^\sharp}(J)$, a contradiction. Hence $i=j$ and we are done. 

Now, assume that $w\neq 1$. Let $w=w_xw_z$, where $w_x$ is a monomial in  $x_1, \ldots, x_n$ and $w_z$ is a monomial in $z^{(i)}_j, 1 \leq i \leq n, 1 \leq j \leq r_i$. Assume that $w_x\neq 1$ and $x_t|w$ for some integer $t$. So we have $x_t\notin C_{i_{a}}$
for some $1\leq a\leq k$. The set
\[
C' = ((C_{i_{a}} \cup \{x_t\}) \setminus V(H_t) ) \cup C^{(t)}_\ell,
\]   
where $1 \leq \ell \leq s_t$ is arbitrary, is a minimal vertex cover of $G(H_1, \ldots, H_n)$ with $u_{C'} <_{\rm lex} u_{C}$.  Let $C' = C_b$.  Then $i_a< b$ and 
$$x_t y_{i_{a}} - \prod_{z^{(t)}_{j'} \not\in C^{(t)}_\ell} z^{(t)}_{j'} y_b \in J_{G(H_1, \ldots, H_n)}$$ whose initial monomial is $x_t y_{i_{a}}$.
Since $x_t| w$ and $u_{C_{i_a}}|h_i$, we may write $wh_i=w'h'$, where $w'=(w/x_t)\prod_{z^{(t)}_{j'} \not\in C^{(t)}_\ell} z^{(t)}_{j'}$ and $h'=(h_i/u_{C_{i_a}})u_{C_b}$. Therefore
$h_j=w'h'$ with $\deg(w'_x)<\deg(w_x)$. Repeating this procedure we obtain
$h_j=uf$, where $u$ is a monomial with $u_x=1$ and $u_z\neq 1$ and $f=u_{C_{\ell_1}}\cdots u_{C_{\ell_k}}$ for some $\ell_1,\ldots,\ell_k$. Let $w''(h'')^*$ be the unique standard monomial in $T$ with respect to $<^\sharp$ and $J$ such that $\pi(w''(h'')^*)=f$, where $w''$ is a monomial in $S$.    

Then $h_j=uf=uw''h''$. Let $h''=u_{C_{s_1}}\cdots u_{C_{s_k}}$  If $x_t|w''$ for some integer $t$, then we have $x_t\notin C_{s_{a}}$ for some $a$ and then $x_t y_{s_{a}}\in \ini_{<^\sharp}(J)$. Since $x_ty_{s_{a}}$ divides $w''(h'')^*$, this implies that  $w''(h'')^*\in\ini_{<^\sharp}(J)$, a contradiction. Hence $w''_x=1$. Note that $(h'')^*$ is a standard monomial as well. So the equality $h_j=uf=u_zw''_zh''$ shows that we may reduce to the case that $w_x=1$. Hence $w=w_z$ with $w_z\neq 1$ and $h_j=w_zh_i$.    

Let, say, $z^{(i)}_j$ divides $w_z$.  In each of $h_j$, $h_i$ and $w$, replace each of $x_{1}, \ldots, x_n$ with $1$ and replace $z^{(i')}_j$ with $1$ for each $i' \neq i$ and for each $1 \leq j \leq r_{i'}$.  Then $h_j$ comes to $$ u_{C^{(i)}_{c_1}}t \cdots u_{C^{(i)}_{c_{p'}}}t\,\left(\prod_{j=1}^{r_i}z_j^{(i)}\right)^{p-p'}$$ and $h_i$ comes to $$u_{C^{(i)}_{c'_1}}t \cdots u_{C^{(i)}_{c'_{p'}}}t\,\left(\prod_{j=1}^{r_i}z_j^{(i)}\right)^{p-p'},$$ and $w$ comes to $v$, where $v$ is a monomial in $z_{1}^{(i)}, \ldots, z_{r_i}^{(i)}$. One has $p' > 0$ and 
\begin{equation}\label{quad}
y^{(i)}_{c_1} \cdots y^{(i)}_{c_{p'}} - v \cdot y^{(i)}_{c'_1} \cdots y^{(i)}_{c'_{p'}} \in J_{H_i}.
\end{equation}

Moreover, since $h_i^*$ and $h_j^*$ are standard monomials in $T$,  $y^{(i)}_{c_1} \cdots y^{(i)}_{c_{p'}}$ and $y^{(i)}_{c'_1} \cdots y^{(i)}_{c'_{p'}}$ are standard monomials in $T_i$, i.e., they do not belong to $\ini_<(J_{H_i})$. Since $J_{H_i}$ has a quadratic Gr\"obner basis with respect to $z^{(i)}_1 > \cdots > z^{(i)}_{r_i}$, as was shown in the proof of \cite[Theorem~3.6]{HHM}, the monomials $u_{C^{(i)}_{c_1}} \cdots u_{C^{(i)}_{c_{p'}}}$  and $u_{C^{(i)}_{c'_1}}\cdots u_{C^{(i)}_{c'_{p'}}}$ belong to the minimal set of monomial generators of $(I_{H_i})^{p'}$. This contradicts to equation (\ref{quad}).      
\end{proof}

The following corollaries are derived from Theorem~\ref{construction}.

\begin{Corollary}\label{cor1}
Let $G'=G(H_1,\ldots,H_n)$, where $G$ is any graph on $n$ vertices and each $H_i$ belongs to one of the following families of graphs:
\begin{itemize}
	\item  [{(a)}] Unmixed chordal graphs.
	\item  [{(b)}]  Cohen-Macaulay bipartite graphs.
	\item  [{(c)}]  Path graphs.
	\item  [{(d)}]  Cameron-Walker graphs whose bipartite graphs are complete graphs. 
\end{itemize}

Then any powers of the vertex cover ideal $I_{G'}$ possesses an order of linear quotients on its minimal monomial generating set and hence it is componentwise linear.	  
\end{Corollary}

\begin{proof}
For any graph $H_i$ belonging to one of the families (a) to (d), the defining ideal of $\Rc(I_{H_i})$  has a quadratic Gr\"obner basis, see \cite[Theorem 2.7]{HHO}, \cite{HHbi} and
 \cite[Corollary 4.7]{HHM}. Hence the desired result follows from Theorem \ref{construction}. 
\end{proof}

Let $G$ be a graph with $V(G)=\{x_1,\ldots,x_n\}$. A {\em cone} on $G$ is a graph $G'$ with $V(G')=V(G)\cup\{y\}$ and 
$E(G')=E(G)\cup\{\{x_i,y\}:\ 1\leq i\leq n\}$, where $y$ is a new vertex. The vertex $y$ is called a {\em universal vertex} of $G'$ and $G'$ is called a {\em cone graph}.

A {\em friendship graph} $F_n$ is a planar graph with $2n + 1$ vertices and $3n$ edges.
It can be constructed by joining $n$ copies of the cycle graph $C_3$ with a common vertex. A {\em fan graph} $F_{1,n}$ is a cone on a path graph with $n$ vertices.

\begin{Corollary}\label{cor2}
Let $G$ be one of the following graphs
\begin{enumerate}
	\item [\em{(a)}] A Cameron-Walker graph with the bipartite partition $(X,Y)$ for which there exists one leaf attached to each vertex in $X$ and at least one pendant triangle attached to each vertex of $Y$.  
	\item [\em{(b)}] A cone graph with a universal vertex $x$ such that $G\setminus \{x\}$ is one of the graphs (a) to (d) in Corollary~\ref{cor1}. Examples of such cone graphs are the fan graphs  $F_{1,n}$, friendship graphs $F_n$ and star graphs. 
\end{enumerate}
Then all powers of the vertex cover ideal $I_{G}$ have linear quotients and hence are componentwise linear.
\end{Corollary}

The following result is an immediate consequence of
Corollary~\ref{cor2} and the characterization of
Cohen-Macaulay Cameron-Walker graphs given in \cite[Theorem 1.3]{HHKO}.

\begin{Corollary}\label{cor3}
	Let $G$ be a Cohen-Macaulay Cameron-Walker graph. Then all power of the vertex cover ideal of $G$ have linear resolutions.
\end{Corollary}

\end{document}